\documentclass[11pt]{amsart}
\usepackage{amsmath,empheq}
\usepackage{paralist}
\usepackage{graphics} 
\usepackage{epsfig} 
\usepackage{graphicx}  \usepackage{epstopdf}
\usepackage[colorlinks=true]{hyperref}
\usepackage{amsfonts}
\usepackage{amssymb}
\usepackage{mathrsfs} 
\usepackage{paralist}  
\hypersetup{urlcolor=blue, citecolor=red} 
\newtheorem{theorem}{Theorem}[section]

\newtheorem{lemma}[theorem]{Lemma}

\theoremstyle{definition}

\newtheorem{remark}{Remark}
\newtheorem{example}{Example}

\newcommand{\N}{\mathbb N}

\newcommand{\R}{\mathbb R}

\newcommand{\essinf}{\mathop{\rm essinf\,}}

\newcommand{\RR}{\mathbb R}
\newcommand{\NN}{\mathbb N}

\newcommand{\Tr}{\texttt{\rmfamily{Tr}}}
\newcommand{\E}{\texttt{\rmfamily{E}}}
\def\div{\mathop{\rm div}}

\title[Nonlinear equations involving...] 
      {Nonlinear equations involving the
       square root of the Laplacian}
\author[V. Ambrosio, G. Molica Bisci and D. D. Repov\v{s}]{}
\subjclass[2010]{Primary: 49J35, 35A15, 35S15; Secondary: 47G20, 45G05.}
 \keywords{Fractional Laplacian, variational methods, multiple solutions.}
 \email{\scshape vincenzo.ambrosio@uniurb.it}
 \email{\scshape gmolica@unirc.it}
 \email{\scshape dusan.repovs@guest.arnes.si}

\begin{document}

\maketitle

\centerline{\scshape Vincenzo Ambrosio}

\medskip

{\small
 \centerline{Dipartimento di Scienze Pure e Applicate (DiSPeA)}
   \centerline{Universit\`{a} degli Studi di Urbino `Carlo Bo'
   Piazza della Repubblica, 13}
   \centerline{ 61029 Urbino, Pesaro e Urbino, Italy}
   }
   
\medskip

\centerline{\scshape Giovanni Molica Bisci}

\medskip

{\small
 \centerline{Dipartimento PAU}
   \centerline{Universit\`a  degli Studi `Mediterranea' di Reggio Calabria}
   \centerline{Salita Melissari - Feo di Vito}
   \centerline{89100 Reggio Calabria, Italy}
}

\medskip

\centerline{\scshape  Du\v{s}an D. Repov\v{s}}
\medskip
{\small
 \centerline{ Faculty of Education, and Faculty of Mathematics and Physics}
   \centerline{University of Ljubljana}
   \centerline{SI-1000 Ljubljana, Slovenia}
}

\bigskip

\centerline{\it Dedicated to Professor Vicen\c{t}iu R\u{a}dulescu with deep esteem and admiration}

\begin{abstract}
In this paper we discuss the existence and non-existence of weak solutions to parametric fractional equations involving the square root of the Laplacian $A_{1/2}$ in a smooth bounded domain $\Omega\subset \R^{n}$ ($n\geq 2$) and with zero Dirichlet boundary conditions. Namely, our simple model is the following equation
\begin{equation*}
\left\{
\begin{array}{ll}
A_{1/2}u=\lambda f(u) & \mbox{ in } \Omega\\
u=0 & \mbox{ on } \partial\Omega.
\end{array}\right.
\end{equation*}
\noindent The existence of at least two non-trivial $L^{\infty}$-bounded weak solutions is established for large value of the parameter $\lambda$, requiring that the nonlinear term $f$ is continuous, superlinear at zero and sublinear at infinity. Our approach is based on variational arguments and a suitable variant of the Caffarelli-Silvestre extension method.
\end{abstract}
\section{Introduction}\label{sec:introduzione}

This paper is concerned with the existence of solutions to nonlinear
problems involving a non-local positive operator: the square root
of the Laplacian in a bounded domain with zero Dirichlet boundary
conditions.\par
 More precisely, from the variational viewpoint, we study
 the existence and non-existence of weak solutions to the following fractional problem
\begin{equation}\label{problema}
\left\{
\begin{array}{ll}
A_{1/2}u=\lambda \beta(x)f(u) & \mbox{ in } \Omega\\
u=0 & \mbox{ on } \partial\Omega,
\end{array}\right.
\end{equation}
where $\Omega$ is an open bounded subset of $\R^n$ ($n\geq 2$) with Lipschitz boundary $\partial \Omega$, $\lambda$ is a positive real parameter, and $\beta:\Omega\to\R$ is a function belonging to $L^{\infty}(\Omega)$ and satisfying
\begin{equation}\label{proprietabeta}
\essinf_{x\in\Omega}\beta(x) > 0.
\end{equation}

 \indent Moreover, the fractional non-local operator $A_{1/2}$ that appears in \eqref{problema} is defined by using the approach developed in the pioneering works of Caffarelli \& Silvestre \cite{CSCPDE}, Caffarelli \& Vasseur \cite{CafVas}, and Cabr\'{e} \& Tan \cite{cabretan}, to which we refer in Section 2 for
the precise mathematical description and properties. We also notice that $A_{1/2}$ which we consider, should not be
confused with the integro-differential operator defined, up to a constant, as
$$(-\Delta)^{1/2} u(x):=
-\int_{\mathbb{R}^{n}}\frac{u(x+y)+u(x-y)-2u(x)}{|y|^{n+1}}\,dy, \quad\, \forall\, x\in \RR^n.$$
 In fact, Servadei \& Valdinoci in \cite{svd} showed that these two operators, although often denoted in the same
way, are really different, with eigenvalues and eigenfunctions behaving
differently (see also Musina \& Nazarov \cite{MuNa}).\par
As pointed out in \cite{cabretan}, the fractions of the Laplacian, such as the previous square root of the Laplacian $A_{1/2}$,
are the infinitesimal generators of L\'{e}vy stable diffusion processes and
appear in anomalous diffusions in plasmas, flames propagation and
chemical reactions in liquids, population dynamics, geophysical fluid
dynamics, and American options in finance. Moreover, a lot of interest has been devoted to elliptic equations involving the fractions of the Laplacian, (see, among others, the papers \cite{A2, Am,Am1, colorado1, colorado2, capella, MaMu, MRadu2, MuPa, tan} as well as \cite{cit1, Cit2,Cit3, Cit4bis, Cit4bisbis, Cit4,Cit5} and the references therein). See also the papers \cite{Ap,pp} for related topics.\par
In our context, regarding the nonlinear term, we assume that $f:\R\to\R$ is continuous, \textit{superlinear at zero}, i.e.
\begin{equation}\label{supf}
\lim_{t\to 0}\frac{f(t)}{t}=0,
\end{equation}
\textit{sublinear at infinity}, i.e.
\begin{equation}\label{condinfinito}
\lim_{|t|\to \infty}\frac{f(t)}{t}=0,
\end{equation}
and such that
\begin{equation}\label{Segno}
\sup_{t\in \R}F(t)>0,
\end{equation}
where
$$
F(t):=\int_0^t f(z)dz,
$$
for any $t\in\R$. Assumptions \eqref{supf} and \eqref{condinfinito} are quite standard in the presence of subcritical terms. Moreover, together with \eqref{Segno}, they guarantee that the number
\begin{equation}\label{c6}
c_f:=\max_{|t|> 0}\frac{|f(t)|}{|t|}
\end{equation}
is well-defined and strictly positive. Furthermore, property \eqref{supf} is a sublinear growth condition at infinity on the nonlinearity $f$ which complements the classical Ambrosetti and Rabinowitz assumption.\par
Here, and in the sequel, we denote by $\lambda_1$ the first eigenvalue of the operator $-\Delta$ in $\Omega$ with homogeneous Dirichlet boundary data, namely the first (simple and positive) eigenvalue of the linear problem
$$
\left\{
\begin{array}{ll}
-\Delta u = \lambda u  & \mbox{ in } \Omega\\
u=0 & \mbox{ on } \partial\Omega.
\end{array}
\right.
$$
\indent The main result of the present paper is an existence theorem for equations driven by the square root of the Laplacian, as stated below.

\begin{theorem}\label{generalkernel0f}
Let $\Omega$ be an open bounded set of $\R^n$ $\mathopen{(}n\geq 2\mathclose{)}$ with Lipschitz boundary $\partial\Omega$, $\beta:\Omega\to\R$ a function satisfying \eqref{proprietabeta}, and $f:\R\to\R$ a continuous function satisfying  \eqref{supf}--\eqref{Segno}. Then the following assertions hold:
\begin{itemize}
\item[$(i)$] problem \eqref{problema} admits only the trivial solution whenever $$0\leq \lambda <\frac{\lambda_1^{1/2}}{c_f\left\|\beta\right\|_{L^{\infty}(\Omega)}} ;$$
\item[$(ii)$] there exists $\lambda^\star>0$ such that \eqref{problema} admits at least two distinct and non-trivial weak solutions $u_{1,\lambda}, u_{2,\lambda}\in L^{\infty}(\Omega) \cap H_0^{1/2}(\Omega)$,
    provided that $\lambda>\lambda^\star$.
\end{itemize}
\end{theorem}

Furthermore, in the sequel we will give additional information about the localization
of the parameter $\lambda^{\star}$. More precisely, by using the notations clarified later on in the paper, we show that
$$
\lambda^{\star}\in \left[\frac{\lambda_1^{1/2}}{c_f\left\|\beta\right\|_{L^{\infty}(\Omega)}},\lambda_{0}\right],
$$
see Remark \ref{bound} for details.\par
Theorem \ref{generalkernel0f} will be proved by applying classical variational techniques to the fractional framework.
More precisely, following \cite{cabretan},  we transform
problem \eqref{problema} to a local problem in one more dimension by using the notion of harmonic extension and the Dirichlet to Neumann map on $\Omega$ (see Section \ref{sec:preliminaries}). By studying
this extended problem with the classical minimization techniques in addition to the Mountain Pass Theorem, we are able to prove the existence of
at least two weak solutions whenever the parameter $\lambda$ is sufficiently large (for instance when $\lambda>\lambda_0$). Finally, the boundedness of the solutions immediately follows from \cite[Theorem 5.2]{cabretan}.\par

We emphasize that Cabr\'e \& Tan in \cite{cabretan} and Tan in \cite{tan3} studied the existence and non-existence of positive solutions for problem \eqref{problema} with power-type nonlinearities, the regularity and an $L^{\infty}$-estimate of weak solutions,
a symmetry result of the Gidas-Ni-Nirenberg type, and a priori estimates
of the Gidas-Spruck type.\par
Along this direction, we look here at the existence of positive $L^{\infty}$-bounded weak solutions on Euclidean balls in presence of sublinear term at infinity. To this end, for every $n\geq 2$ and $r>0$, set
$$
\zeta(n,r):=
\frac{8r^2}{\displaystyle r^2+\displaystyle 4\min_{\sigma\in \Sigma_n}z_n(\sigma)},\quad\quad {\rm with}\quad\quad\Sigma_n:=\left(\frac{1}{2^{1/n}},1\right)
$$
where
$$
z_n(\sigma):=\frac{1-\sigma^n}{(2\sigma^n-1)(1-\sigma)^2},\quad\forall\sigma\in \Sigma_n.
$$

\indent With the above notations, a special case of Theorem \ref{generalkernel0f} reads as follows.

\begin{theorem}\label{boccia}
Let $r>0$ and denote
$$
\Gamma^0_r:=\{(x,0)\in \partial \R^{n+1}_+:|x|<r\},
$$
where $\partial \R^{n+1}_+:=\R^n\times (0,+\infty)$ and $n\geq 2$. Moreover, let $f:[0,+\infty)\rightarrow \R$ be a continuous non-negative and non-identically zero function such that
$$
\lim_{t\to 0^+}\frac{f(t)}{t}=\lim_{t\to +\infty}\frac{f(t)}{t}=0,
$$
\noindent with
\begin{equation}\label{disball}
\min_{t\in S}\frac{t^2}{F(t)}<\zeta(n,r),
\end{equation}
where
$$
S:=\{t>0:F(t)>0\}.
$$
\indent Then the following nonlocal problem
\begin{equation}\label{problema24}
\left\{
\begin{array}{ll}
A_{1/2}u= f(u) & \mbox{\rm in } \Gamma^0_r\\
u> 0 & \mbox{\rm on } \Gamma^0_r\\
u=0 & \mbox{\rm on } \partial\Gamma^0_r
\end{array}\right.
\end{equation}
admits at least two distinct weak solutions $u_{1,\lambda}, u_{2,\lambda}\in L^{\infty}(\Gamma^0_r) \cap H_0^{1/2}(\Gamma^0_r)$.
\end{theorem}

The structure of this paper is as follows. After presenting the functional space related to problem \eqref{problema} together with its basic properties (Section \ref{sec:preliminaries}), we show via direct computations that for a determined right neighborhood of $\lambda$, the zero solution is the unique one (Section \ref{sec:non-exist}). In Section \ref{sec:multiple} we prove the existence of two weak solutions for $\lambda$ bigger than a certain $\lambda^\star$: the first one is obtained via direct minimization, the second one via the Mountain Pass Theorem. Specific bounds for $\lambda^\star$ are obtained in Remark \ref{bound}.\par

We refer to the recent book \cite{MRS}, as well as \cite{valpal}, for the abstract variational setting used in the present paper. See the recent very nice papers \cite{M0,M1} of Kuusi, Mingione \& Sire on nonlocal fractional problems.
 
\section{Preliminaries}\label{sec:preliminaries}
In this section we briefly recall the definitions of the functional space setting, first introduced in \cite{cabretan}.
The reader familiar with this topic may skip this section and go directly to the next one.
\subsection{Fractional Sobolev spaces}
The power $A_{1/2}$ of the Laplace operator $-\Delta$ in a bounded domain $\Omega$ with zero boundary conditions is defined through the spectral decomposition using the powers of the eigenvalues of the original operator.\par
 Hence, according to classical results on positive operators in $\Omega$, if $\{\varphi_j,\lambda_j\}_{j\in \NN}$ are the eigenfunctions and eigenvalues of the usual linear Dirichlet problem
\begin{equation}\label{problema autovalori}
\left\{\begin{array}{ll}
-\Delta u=\lambda u & \mbox{in } \Omega\\
u=0 & \mbox{on } \partial\Omega,
\end{array}
\right.
\end{equation}
then $\{\varphi_j,\lambda_j^{1/2}\}_{j\in\NN}$ are the eigenfunctions and eigenvalues of the corresponding fractional one:
\begin{equation}\label{problema autovalori2}
\left\{\begin{array}{ll}
A_{1/2} u=\lambda u & \mbox{in } \Omega\\
u=0 & \mbox{on } \partial\Omega.
\end{array}
\right.
\end{equation}

\indent We
repeat each eigenvalue of $-\Delta$ in $\Omega$
 with zero Dirichlet boundary conditions
according to its (finite) multiplicity:
$$0<\lambda_{1}<\lambda_{2}\le \dots \le \lambda_{j}\le \lambda_{j+1}\le \dots$$
and
$\lambda_j\rightarrow +\infty$ as $j\rightarrow +\infty$. Moreover, we can suppose that the eigenfunctions $\{\varphi_j\}_{j\in \N}$  are normalized as follows:
$$
\int_\Omega |\nabla\varphi_j(x)|^2dx=\lambda_j\int_\Omega |\varphi_j(x)|^2dx=\lambda_j,\quad\forall\, j\in \N
$$
and
$$
\int_\Omega \nabla\varphi_i(x)\cdot\nabla\varphi_j(x)dx=\int_\Omega \varphi_i(x)\varphi_j(x)dx=0,\quad\forall\, i\neq j.
$$
Finally, standard regularity arguments ensure that $\varphi_j\in C^{2}(\bar \Omega)$, for every $j\in \N$.

\indent The operator $A_{1/2}$ is well-defined on the Sobolev space
$$
H_0^{1/2}(\Omega):=\left\{u\in L^2(\Omega): u=\displaystyle\sum_{j=1}^{\infty}a_j\varphi_j\,\,{\rm and}\,\,\displaystyle\sum_{j=1}^{\infty} a_j^2\lambda_j^{1/2}<+\infty\right\},
$$
endowed by the norm
$$
\|u\|_{H_0^{1/2}(\Omega)}:=\displaystyle\left(\sum_{j=1}^{\infty} a_j^2\lambda_j^{1/2}\right)^{1/2},
$$
and has the following form
$$
A_{1/2}u=\sum_{j=1}^{\infty}a_j\lambda_j^{1/2}\varphi_j, \, \mbox{ where } a_{j}:=\int_{\Omega} u(x) \varphi_{j}(x) dx.
$$
\subsection{The extension problem}\label{subsec:b}

Associated to the bounded domain $\Omega$, let us consider the cylinder
$$
\mathcal{C}_\Omega:=\{(x,y): x\in \Omega,\,\, y>0\}\subset \RR_+^{n+1},
$$
and denote by $\partial_L \mathcal{C}_\Omega:=\partial\Omega\times [0,+\infty)$ its lateral boundary.\par
For a function $u\in H_0^{1/2}(\Omega)$, define the harmonic extension $\E(u)$ to the cylinder $\mathcal{C}_\Omega$ as the solution of the problem
\begin{align}\label{1.1}
 \left\{
\begin{array}{ll}
\displaystyle{\div(\nabla \E(u))=  0 } \quad &\mbox{\rm in } \mathcal{C}_\Omega\\
\displaystyle{\E(u)=0} \quad &\mbox{\rm on } \partial_L\mathcal{C}_\Omega\\
\Tr(\E(u))=u \quad &\mbox{\rm on } \Omega,
\end{array}\right.
\end{align}
where
$$
\Tr(\E(u))(x):=\E(u)(x,0),\quad\, \forall\, x\in \Omega.
$$
\indent The extension function $\E(u)$ belongs to the Hilbert space
$$
X^{1/2}_0(\mathcal{C}_\Omega):=\left\{w\in L^2(\mathcal{C}_\Omega):w=0\,\,{\rm on }\,\, \partial_L \mathcal{C}_\Omega,\,\, \int_{\mathcal{C}_\Omega}|\nabla w(x,y)|^2\,dxdy<+\infty\right\},
$$
with the standard norm
$$
\|w\|_{X^{1/2}_0(\mathcal{C}_\Omega)}:=\left(\int_{\mathcal{C}_\Omega}|\nabla w(x,y)|^2\,dxdy\right)^{1/2}.
$$
Hence the space $X^{1/2}_0(\mathcal{C}_\Omega)$ is defined by
$$
X^{1/2}_0(\mathcal{C}_\Omega)=\left\{w\in H^1(\mathcal{C}_\Omega): w=0\,\,{\rm on }\,\, \partial_L \mathcal{C}_\Omega\right\},
$$
and can be characterized as follows
$$
X^{1/2}_0(\mathcal{C}_\Omega)=\left\{w\in L^2(\mathcal{C}_\Omega):w=\displaystyle\sum_{j=1}^{\infty}b_j\varphi_je^{-\lambda_j^{1/2}y}\,\,{\rm with}\,\,\displaystyle\sum_{j=1}^{\infty} b_j^2\lambda_j^{1/2}<+\infty\right\},
$$
see \cite[Lemma 2.10]{cabretan}.\par
\indent In our framework, a crucial role between the spaces $X^{1/2}_0(\mathcal{C}_\Omega)$ and $H_0^{1/2}(\Omega)$ is played by trace operator $\Tr: X^{1/2}_0(\mathcal{C}_\Omega)\rightarrow H_0^{1/2}(\Omega)$ given by
$$
\Tr(w)(x):=w(x,0),\,\quad \forall\, x\in \Omega.
$$
\indent The trace operator is a continuous map (see \cite[Lemma 2.6]{cabretan}), and gives a lot of information,
which we recall in the sequel. We also notice that
$$
H^{1/2}_0(\Omega)=\{u\in L^2(\Omega):u=\Tr(w),\,{\rm for\, some}\,w\in X^{1/2}_0(\mathcal{C}_\Omega)\}\subset H^{1/2}(\Omega),
$$
and that the extension operator $\E:H_0^{1/2}(\Omega)\rightarrow X^{1/2}_0(\mathcal{C}_\Omega)$ is an isometry i.e.
$$
\|\E(u)\|_{X^{1/2}_0(\mathcal{C}_\Omega)}=\|u\|_{H_0^{1/2}(\Omega)},
$$
for every $u\in H_0^{1/2}(\Omega)$. Here, $H^{1/2}(\Omega)$ denotes the Sobolev space of order $1/2$, defined as
$$
H^{1/2}(\Omega):=\left\{u\in L^2(\Omega):\int_{\Omega\times\Omega}\frac{|u(x)-u(y)|^2}{|x-y|^{n+1}}dxdy<+\infty\right\},
$$
with the norm
$$
\|u\|_{H^{1/2}(\Omega)}:=\left(\int_{\Omega\times\Omega}\frac{|u(x)-u(y)|^2}{|x-y|^{n+1}}dxdy+\int_\Omega|u(x)|^2dx\right)^{1/2}.
$$
Next, we have the following trace inequality
\begin{equation}\label{traceine}
\|\Tr(w)\|_{H^{1/2}_0(\Omega)}\leq \|w\|_{X^{1/2}_0(\mathcal{C}_\Omega)},
\end{equation}
for every $w\in X^{1/2}_0(\mathcal{C}_\Omega)$. Before concluding this subsection, we recall the embedding
properties of~$\Tr(X^{1/2}_0(\Omega))$ into the usual Lebesgue spaces; see
\cite[Lemmas 2.4 and 2.5]{cabretan}.\par
 More precisely, the embedding $j:\Tr(X_0^{1/2}(\mathcal{C}_\Omega))\hookrightarrow
L^{\nu}(\Omega)$ is continuous for any $\nu\in [1,2^{\sharp}]$, and is
compact whenever $\nu\in [1,2^{\sharp})$, where $2^{\sharp}:=2n/(n-1)$ denotes the \textit{fractional critical Sobolev exponent}.\par
Thus, if $\nu\in [1,2^{\sharp}]$, then there exists a positive constant $c_\nu$ (depending on $\nu$, $n$ and the Lebesgue measure of $\Omega$, denoted by $|\Omega|$) such that
\begin{equation}\label{constants}
\left(\displaystyle\int_{\Omega}|\Tr(w)(x)|^{\nu}dx\right)^{1/\nu}\leq c_\nu \left(\displaystyle\int_{\mathcal{C}_\Omega}|\nabla w(x,y)|^2\,dxdy\right)^{1/2},
\end{equation}
for every $w\in X_0^{1/2}(\mathcal{C}_\Omega)$. From now on, for every $q\in [1,\infty]$, $\|\cdot\|_{L^{q}(\Omega)}$ denotes the usual norm of the Lebesgue space $L^q(\Omega)$.\par
As already said, we will consider the square root of the Laplacian,
defined according to the following procedure (see, for instance, the papers \cite{colorado1, colorado2, cabretan}).
By using the extension $\E(u)\in X^{1/2}_0(\mathcal{C}_\Omega)$ of the function $u\in H_0^{1/2}(\Omega)$, we can define the fractional operator $A_{1/2}$ in $\Omega$, acting on $u$, as follows$:$
$$
A_{1/2}u(x):=-\lim_{y\rightarrow 0^+}\frac{\partial \E(u)}{\partial y}(x,y),\quad\,\forall\, x\in \Omega
$$
i.e.
$$
A_{1/2}u(x)=\displaystyle\frac{\partial \E(u)}{\partial \nu}(x),\quad\,\forall\, x\in \Omega
$$
where $\nu$ is the unit outer normal to $\mathcal{C}_{\Omega}$ at $\Omega\times\{0\}$.\par
\subsection{Weak solutions}
\indent Assume that $f:\R\rightarrow\R$ is a subcritical function and $\lambda>0$ is fixed. We say that a function $u=\Tr(w)\in H^{1/2}_0(\Omega)$ is a \textit{weak solution} of the problem \eqref{problema} if $w\in X^{1/2}_0(\mathcal{C}_\Omega)$ weakly solves
\begin{equation}\label{2}
\left\{\begin{array}{ll} \displaystyle{ -\div(\nabla w)=  0 } &\textrm{ in } \mathcal{C}_\Omega\\
\,\,\,\,\,\,\,\,\,\,\,\,\,\,\,\,\,\displaystyle{w=0} &\textrm{ on } \partial_L\mathcal{C}_\Omega\\
\displaystyle\frac{\partial w}{\partial \nu}=\lambda    \beta(x)f(\Tr(w)) & \text{ on } \Omega,
\end{array}\right.
\end{equation}
i.e.
\begin{equation}\label{EL}
\int_{\mathcal{C}_\Omega}\langle \nabla w,\nabla \varphi \rangle dxdy=\lambda\int_{\Omega}\beta(x)f(\Tr (w)(x))\Tr(\varphi)(x)dx,
\end{equation}
for every $\varphi\in X^{1/2}_0(\mathcal{C}_\Omega)$.\par
As direct computations prove, equation \eqref{EL} represents the variational formulation of \eqref{2} and the energy functional $\mathcal J_{\lambda}:X^{1/2}_0(\mathcal{C}_\Omega)\to \R$ associated with \eqref{EL} is defined by
\begin{equation}\label{JKlambda}
\begin{aligned}
\mathcal J_{\lambda}(w)& :=\frac{1}{2}\int_{\mathcal{C}_\Omega}|\nabla w(x,y)|^2\,dxdy\\
& \qquad \qquad \qquad \qquad
-\lambda\int_{\Omega}\beta(x)F(\Tr (w)(x))dx,
\end{aligned}
\end{equation}
for every $w\in X^{1/2}_0(\mathcal{C}_\Omega)$.\par
Indeed, as it can be easily seen, under our assumptions on the nonlinear term, the functional
$\mathcal J_{\lambda}$ is well-defined and of class $C^1$ in $X^{1/2}_0(\mathcal{C}_\Omega)$. Moreover, its critical points
are exactly the weak solutions of the problem \eqref{2}.\par
 Thus the traces of critical points of $\mathcal J_{\lambda}$ are the weak solutions to
problem~\eqref{problema}. According to the above remarks,
we will use critical point methods  in order to prove Theorems \ref{generalkernel0f} and \ref{boccia}.
\section{The main theorem: Non-existence for small $\lambda$}\label{sec:non-exist}

Let us prove assertion $(i)$ of Theorem \ref{generalkernel0f}.

Arguing by contradiction, suppose that there exists a weak solution $w_0\in\break X^{1/2}_0(\mathcal{C}_\Omega)\setminus\{0\}$ to problem \eqref{problema}, i.e.
\begin{equation}\label{problemaK0f45}
\int_{\mathcal{C}_\Omega}\langle \nabla w_0,\nabla \varphi \rangle dxdy=\lambda\int_{\Omega}\beta(x)f(\Tr (w_0)(x))\Tr(\varphi)(x)dx,
\end{equation}
for every $\varphi\in X^{1/2}_0(\mathcal{C}_\Omega)$.\par
Testing \eqref{problemaK0f45} with $\varphi:=w_0$, we have
\begin{equation}\label{au34}
\|w_0\|^2_{X^{1/2}_0(\mathcal{C}_\Omega)} =\lambda\int_{\Omega}\beta(x)f(\Tr (w_0)(x))\Tr(w_0)(x)dx,
\end{equation}
and it follows that
\begin{equation}\label{au345}
\begin{split}
\int_{\Omega}\beta(x)f(\Tr (w_0)(x))\Tr(w_0)(x)dx & \leq  \int_{\Omega}\beta(x)\left|f(\Tr (w_0)(x))\Tr(w_0)(x)\right|dx\\
                                       & \leq  c_f\|\beta\|_{L^{\infty}(\Omega)} \|\Tr(w_0)\|_{L^2(\Omega)}^2\\
                                       & \leq  \frac{c_f}{\lambda_1^{1/2}}\|\beta\|_{L^{\infty}(\Omega)}\|w_0\|^2_{X^{1/2}_0(\mathcal{C}_\Omega)}.
\end{split}
\end{equation}
In the last inequality we have used the following fact
$$
\lambda_{1}^{1/2}=\min_{w\in X^{1/2}_0(\mathcal{C}_\Omega)\setminus\{0\}}\frac{\displaystyle\int_{\mathcal{C}_\Omega}|\nabla w(x,y)|^2\,dxdy}{\displaystyle\int_{\Omega}|\Tr(w)(x)|^{2}dx}\leq \frac{\displaystyle\int_{\mathcal{C}_\Omega}|\nabla w_0(x,y)|^2\,dxdy}{\displaystyle\int_{\Omega}|\Tr(w_0)(x)|^{2}dx},
$$
and the trace inequality \eqref{traceine}. By \eqref{au34}, \eqref{au345} and the assumption on $\lambda$ we get
\begin{equation*}
\|w_0\|^2_{X^{1/2}_0(\mathcal{C}_\Omega)}  \leq \lambda \frac{c_f}{\lambda_1^{1/2}}\|\beta\|_{L^{\infty}(\Omega)} \|w_0\|^2_{X^{1/2}_0(\mathcal{C}_\Omega)} < \|w_0\|^2_{X^{1/2}_0(\mathcal{C}_\Omega)} ,
\end{equation*}
clearly a contradiction.
 
\section{The main theorem: Multiplicity}\label{sec:multiple}
\subsection{The variational setting}
 The aim of this section is to prove that, under natural assumptions on the nonlinear term $f$, weak solutions to problem \eqref{problema} below do exist. Our approach to determine multiple solutions to \eqref{problema} consists of applying classical variational methods to the functional $\mathcal J_{\lambda}$. To this end, we write $\mathcal J_{\lambda}$ as
$$
\mathcal J_{\lambda}(w)= \Phi(w)-\lambda \Psi(w),
$$
where
$$
\Phi(w):=\frac 1 2 \left\|w\right\|^2_{X^{1/2}_0(\mathcal{C}_\Omega)},
$$
while
$$
\Psi(w):= \int_{\Omega} \beta(x)F(\Tr(w)(x)) dx,
$$
for every $w\in {X^{1/2}_0(\mathcal{C}_\Omega)}$. Clearly, the functional $\Phi$ and $\Psi$ are Fr\'echet differentiable.

Moreover, the functional $\mathcal J_{\lambda}$ is weakly lower semicontinuous on $X_0^{1/2}(\mathcal{C}_\Omega)$. Indeed, the application $$w\mapsto\int_{\Omega}\beta(x)F(\Tr(w)(x))dx$$ is
continuous in the weak topology of $X_0^{1/2}(\mathcal{C}_\Omega)$.\par
We prove this regularity result as follows. Let
$\{w_{j}\}_{j\in \NN}$ be a sequence in $X_0^{1/2}(\mathcal{C}_\Omega)$ such that
$w_{j}\rightharpoonup w_\infty$ weakly in $X_0^{1/2}(\mathcal{C}_\Omega)$. Then, by using Sobolev embedding results
and \cite[Theorem~IV.9]{brezis}, up to a
subsequence, $\{\Tr(w_{j})\}_{j\in \N}$ strongly converges to $\Tr(w_\infty)$ in $L^{\nu}(\Omega)$
and almost everywhere (a.e.) in $\Omega$ as $j\to +\infty$, and it is dominated by some
function $\kappa_\nu\in L^{\nu}(\Omega)$ i.e.
\begin{equation}\label{dominata}
|\Tr(w_{j})(x)|\leq \kappa_\nu(x)\quad \mbox{a.e.}\,\, x\in \Omega\,\,\,\, \mbox{for any}\, j\in \NN
\end{equation}
for any $\nu\in [1,2^{\sharp})$.\par
Due to \eqref{condinfinito}, there exists $c>0$ such that
\begin{equation}\label{crescita}
|f(t)|\le c(1+|t|),\,\,\,\,\,(\forall\,t\in \RR).
\end{equation}

It then follows by the continuity of $F$ and \eqref{crescita} that
$$F(\Tr(w_{j})(x)) \to F(\Tr(w_{\infty})(x))\,\,\, \mbox{a.e.}\,\, x\in \Omega$$
as $j\to +\infty$ and
$$\left|F(\Tr(w_{j})(x))\right|\leq c\left( |\Tr(w_{j})(x)|+\frac{1}{2}|\Tr(w_{j})(x)|^2\right)\leq c\left( \kappa_1(x)+\frac{1}{2} \kappa_2(x)^2\right)\in L^1(\Omega)$$
a.e. $x\in \Omega$ and for any $j\in \NN$.\par
 Hence, by applying the
Lebesgue Dominated Convergence Theorem in $L^1(\Omega)$, we have
that
$$\int_\Omega \beta(x)F(\Tr(w_{j})(x))\,dx \to \int_\Omega \beta(x)F(\Tr(w_{\infty})(x))\,dx$$
as $j\to +\infty$, that is
the map  $$w\mapsto\int_{\Omega}\beta(x)F(\Tr(w_{j})(x))dx$$
is continuous from $X_0^{1/2}(\mathcal{C}_\Omega)$ with the weak topology to $\RR$.\par
 On the other hand, the map
$$w\mapsto \displaystyle\int_{\mathcal{C}_\Omega}|\nabla w(x,y)|^2\,dxdy$$ is lower
semicontinuous in the weak topology of $X_0^{1/2}(\mathcal{C}_\Omega)$.\par
 Hence, the
functional $\mathcal J_{\lambda}$ is lower semicontinuous in the weak
topology of\break $X_0^{1/2}(\mathcal{C}_\Omega)$.\par
 
\subsection{Sub-quadraticity of the potential}
Let us prove that, under the hypotheses \eqref{supf} and \eqref{condinfinito}, one has
\begin{equation}\label{sottoquadracita}
\lim_{\left\|w\right\|_{X^{1/2}_0(\mathcal{C}_\Omega)}\to 0}\frac{\Psi(w)}{\left\|w\right\|^2_{X^{1/2}_0(\mathcal{C}_\Omega)}} =0\,\,\,\,\,\,{\rm and}\,\,\,\,\,\, \lim_{\left\|w\right\|_{X^{1/2}_0(\mathcal{C}_\Omega)}\to \infty}\frac{\Psi(w)}{\left\|w\right\|^2_{X^{1/2}_0(\mathcal{C}_\Omega)}}=0.
\end{equation}
\indent Fix $\varepsilon>0$. In view of \eqref{supf} and \eqref{condinfinito}, there exists $\delta_\varepsilon\in \mathopen{(}0,1\mathclose{)}$ such that
\begin{equation}\label{stimaf}
|f(t)| \leq \frac{\varepsilon}{\left\|\beta\right\|_{L^{\infty}(\Omega)}}|t|,
\end{equation}
for all $0<|t|\leq \delta_\varepsilon$ and $|t|\geq \delta_\varepsilon^{-1}$.

\indent Let us fix $q\in\mathopen{(}2,2^*\mathclose{)}$. Since the function $$t\mapsto \frac{|f(t)|}{|t|^{q-1}}$$ is bounded on $[\delta_\varepsilon, \delta_\varepsilon^{-1}]$, for some $m_\varepsilon>0$ and for every $t\in \R$ one has
\begin{equation}\label{stimavaloreassolutof}
|f(t)| \leq \frac{\varepsilon}{\left\|\beta\right\|_{L^{\infty}(\Omega)}}|t| + m_\varepsilon |t|^{q-1}.
\end{equation}
\indent As a byproduct, inequality \eqref{stimavaloreassolutof}, in addition to \eqref{constants}, yields
\begin{align*}
|\Psi(w)| & \leq \int_\Omega \beta(x)|F(\Tr(w)(x))| dx\\
          & \leq \int_\Omega \beta(x) \left(\frac{\varepsilon}{2\left\|\beta\right\|_{L^{\infty}(\Omega)}} |\Tr(w)(x)|^2 + \frac{m_\varepsilon}{q} |\Tr(w)(x)|^q\right)dx\\
          & \leq \int_\Omega \left(\frac{\varepsilon}{2}|\Tr(w)(x)|^2 + \frac{m_\varepsilon}{q}\beta(x) |\Tr(w)(x)|^q\right)dx\\
					& \leq \frac{\varepsilon}{2} \left\|\Tr(w)\right\|_{L^{2}(\Omega)}^2 + \frac{m_\varepsilon}{q}\left\|\beta\right\|_{L^{\infty}(\Omega)} \left\|\Tr(w)\right\|_{L^{q}(\Omega)}^q\\
& \leq \frac{\varepsilon}{2} c_2^2\left\|w\right\|^2_{X^{1/2}_0(\mathcal{C}_\Omega)} + \frac{m_\varepsilon}{q}c_q^q \left\|\beta\right\|_{L^{\infty}(\Omega)} \left\|w\right\|^q_{X^{1/2}_0(\mathcal{C}_\Omega)},
\end{align*}
for every $w\in {X^{1/2}_0(\mathcal{C}_\Omega)}$.\par
\indent Therefore, it follows that for every $w\in X^{1/2}_0(\mathcal{C}_\Omega)\setminus\{0\}$,
$$
0 \leq \frac{|\Psi(w)|}{\left\|w\right\|^2_{X^{1/2}_0(\mathcal{C}_\Omega)}} \leq \frac{\varepsilon}{2}c_2^2 + \frac{m_\varepsilon}{q}\left\|\beta\right\|_{L^{\infty}(\Omega)} c_q^q\left\|w\right\|^{q-2}_{X^{1/2}_0(\mathcal{C}_\Omega)}.
$$
Since $q>2$ and $\varepsilon$ is arbitrary, the first limit of \eqref{sottoquadracita} turns out to be zero.

\indent Now, if $r\in \mathopen{(}1,2\mathclose{)}$, due to the continuity of $f$, there also exists a number $M_\varepsilon >0$ such that
$$
\frac{|f(t)|}{|t|^{r-1}} \leq M_\varepsilon,
$$
for all $t\in[\delta_\varepsilon,\delta_\varepsilon^{-1}]$, where $\varepsilon$ and $\delta_\varepsilon$ are the previously introduced numbers.\par
 The above inequality, together with \eqref{stimaf}, yields
$$
|f(t)| \leq\frac{\varepsilon}{\left\|\beta\right\|_{L^{\infty}(\Omega)}} |t| + M_\varepsilon |t|^{r-1}
$$
for each $t\in \R$ and hence
\begin{align*}
|\Psi(w)| & \leq \int_\Omega \beta(x)|F(\Tr(w)(x))| dx\\
          & \leq \int_\Omega \beta(x) \left(\frac{\varepsilon}{2\left\|\beta\right\|_{L^{\infty}(\Omega)}} |\Tr(w)(x)|^2 + \frac{M_\varepsilon}{r} |\Tr(w)(x)|^r\right)dx\\
          & \leq \int_\Omega \left(\frac{\varepsilon}{2}|\Tr(w)(x)|^2 + \frac{M_\varepsilon}{r}\beta(x) |\Tr(w)(x)|^r\right)dx\\
					& \leq \frac{\varepsilon}{2} \left\|\Tr(w)\right\|_{L^{2}(\Omega)}^2 + \frac{M_\varepsilon}{r}\left\|\beta\right\|_{L^{\infty}(\Omega)} \left\|\Tr(w)\right\|_{L^{r}(\Omega)}^r\\
& \leq \frac{\varepsilon}{2}c_2^2 \left\|w\right\|^2_{X^{1/2}_0(\mathcal{C}_\Omega)} + \frac{M_\varepsilon}{r}\left\|\beta\right\|_{L^{\infty}(\Omega)} c_r^r\left\|w\right\|^r_{X^{1/2}_0(\mathcal{C}_\Omega)},
\end{align*}
for each $w\in X^{1/2}_0(\mathcal{C}_\Omega)$.\par
 Therefore, it follows that for every $w\in X^{1/2}_0(\mathcal{C}_\Omega)\setminus\{0\}$,
\begin{equation}\label{stimaadinfinito}
0 \leq \frac{|\Psi(w)|}{\left\|w\right\|^2_{X^{1/2}_0(\mathcal{C}_\Omega)}} \leq \frac{\varepsilon}{2}c_2^2 + \frac{M_\varepsilon}{r}\left\|\beta\right\|_{L^{\infty}(\Omega)} c_r^r\left\|w\right\|^{r-2}_{X^{1/2}_0(\mathcal{C}_\Omega)}.
\end{equation}

\noindent Since $\varepsilon$ can be chosen as small as we wish and $r\in \mathopen{(}1,2\mathclose{)}$, taking the limit for $\left\|w\right\|_{X^{1/2}_0(\mathcal{C}_\Omega)}\to+\infty$ in \eqref{stimaadinfinito}, we have proved the second limit of \eqref{sottoquadracita}.

\subsection{The Palais-Smale condition}
For the sake of completeness, we recall that, if $E$ is a real Banach space, a $C^1$-functional $J:E\to\R$ is said to satisfy the Palais-Smale condition at level $\mu\in\R$ when
\begin{itemize}
\item[$\textrm{(PS)}_{\mu}$] {\it Every sequence $\{z_{j}\}_{j\in \N} \subset E$  such that
$$
J(z_j)\to \mu \quad{\it and}\quad \|J'(z_j)\|_{E^*}\to 0,
$$
when $j\rightarrow +\infty$, possesses a convergent subsequence in $E$.}
\end{itemize}

Here $E^*$ denotes the topological dual of $E$. We say that $J$ satisfies the \textit{Palais-Smale condition} ($(\rm PS)$ in short) if $\textrm{(PS)}_{\mu}$ holds for every $\mu\in \R$.

\begin{lemma}\label{cps}
Let $f:\R \to \R$ be a continuous function satisfying conditions \eqref{supf} and \eqref{condinfinito}. Then for every $\lambda>0$, the functional $\mathcal J_{\lambda}$ is bounded from below, coercive and satisfies $(\rm PS)$.
\end{lemma}
\begin{proof}

Fix $\lambda>0$ and $0<\varepsilon<1/\lambda c_{2}^{2}$.
Due to \eqref{stimaadinfinito}, one has
\begin{align*}
\mathcal J_{\lambda}(w) &\geq \frac{1}{2}\|w\|^2_{X^{1/2}_0(\mathcal{C}_\Omega)} - \lambda\int_\Omega \beta(x)|F(\Tr(w)(x))|dx \\
                        &\geq \frac{1}{2}\|w\|^2_{X^{1/2}_0(\mathcal{C}_\Omega)} - \lambda c_2^2\frac{\varepsilon}{2}\left\|w\right\|^2_{X^{1/2}_0(\mathcal{C}_\Omega)} -\lambda\frac{M_\varepsilon}{r}\left\|\beta\right\|_{L^{\infty}(\Omega)} c_r^r\left\|w\right\|^{r}_{X^{1/2}_0(\mathcal{C}_\Omega)}\\
                        &=\frac{1}{2}\left(1-{\lambda c_2^2}{\varepsilon}\right)\|w\|^2_{X^{1/2}_0(\mathcal{C}_\Omega)}-\lambda\frac{M_\varepsilon}{r}\left\|\beta\right\|_{L^{\infty}(\Omega)} c_r^r\left\|w\right\|^{r}_{X^{1/2}_0(\mathcal{C}_\Omega)},
\end{align*}
for every $w\in X^{1/2}_0(\mathcal{C}_\Omega)$. Then the functional $\mathcal J_{\lambda}$ is bounded from below and coercive.

Now, let us prove that $\mathcal J_{\lambda}$ satisfies $\textrm{(PS)}_{\mu}$ for $\mu\in\R$. To this end, let $\{w_{j}\}_{j\in \N}\subset X^{1/2}_0(\mathcal{C}_\Omega)$ be a Palais-Smale sequence, i.e.
$$
\mathcal J_{\lambda}(w_j)\to \mu \quad{\rm and}\quad \|\mathcal J'_{\lambda}(w_j)\|_{*}\to 0,
$$
as $j\rightarrow +\infty$ where, we set
$$\|\mathcal J'_{\lambda}(w_j)\|_{*}:=\sup\Big\{\big|\langle\,\mathcal J_{\lambda}'(w_j),\varphi \rangle
\big|: \; \varphi\in X^{1/2}_0(\mathcal{C}_\Omega),\,{\rm and}\, \|\varphi\|_{X^{1/2}_0(\mathcal{C}_\Omega)}=1\Big\}.$$
\indent Taking into account the coercivity of $\mathcal J_{\lambda}$, the sequence $\{w_{j}\}_{j\in \N}$ is necessarily bounded in $X^{1/2}_0(\mathcal{C}_\Omega)$. Since $X^{1/2}_0(\mathcal{C}_\Omega)$ is reflexive, we can extract a subsequence, which for simplicity we still denote $\{w_{j}\}_{j\in \N}$, such that
$w_{j}\rightharpoonup w_\infty$ in $X^{1/2}_0(\mathcal{C}_\Omega)$, i.e.,

\begin{equation}\label{conv0}
\int_{\mathcal{C}_\Omega}\langle \nabla w_j,\nabla \varphi \rangle dxdy\to
\int_{\mathcal{C}_\Omega}\langle \nabla w_{\infty},\nabla \varphi \rangle dxdy,
\end{equation}
as $j\to+\infty$, for any $\varphi\in X^{1/2}_0(\mathcal{C}_\Omega)$.

We will prove that $\{w_{j}\}_{j\in \N}$ strongly converges to $w_\infty\in X^{1/2}_0(\mathcal{C}_\Omega)$. One has
\begin{equation}\label{jj}
\langle \Phi'(w_{j}),w_{j}-w_\infty\rangle = \langle \mathcal J_{\lambda}'(w_j),w_j-w_\infty\rangle + \lambda\int_{\Omega}\beta(x)f(\Tr(w_j)(x))\Tr(w_j-w_\infty)(x)dx,
\end{equation}
where
\begin{eqnarray*}
\langle \Phi'(w_{j}),w_{j}-w_\infty\rangle  &=& \displaystyle\int_{\mathcal{C}_\Omega}|\nabla w_j(x,y)|^2\,dxdy\\
               &&\qquad\qquad- \int_{\mathcal{C}_\Omega}\langle \nabla w_j,\nabla w_\infty \rangle dxdy.
\end{eqnarray*}

Since $\|\mathcal J_{\lambda}'(w_j)\|_{*}\to 0$ and the sequence $\{w_j-w_\infty\}_{j\in \N}$ is bounded in $X^{1/2}_0(\mathcal{C}_\Omega)$, taking account of the fact that $|\langle\mathcal J_{\lambda}'(w_j),w_j-w_\infty\rangle|\leq\|\mathcal J_{\lambda}'(w_j)\|_{*}\|w_j-w_\infty\|_{X^{1/2}_0(\mathcal{C}_\Omega)}$, one has
\begin{eqnarray}\label{j2}
\langle \mathcal J_{\lambda}'(w_j),w_j-w_\infty\rangle\to 0
\end{eqnarray}
as $j\to+\infty$.

Next, setting
$$
I:=\int_{\Omega}\beta(x)|f(\Tr(w_j)(x))||\Tr(w_j-w_\infty)(x)|dx,
$$
 one has by \eqref{stimavaloreassolutof} and H\"{o}lder's inequality
\begin{align*}
 I &\leq \varepsilon\int_\Omega |\Tr(w_j)(x)||\Tr(w_j-w_\infty)(x)|dx \\
                                                & \quad+ m_\varepsilon\left\|\beta\right\|_{L^{\infty}(\Omega)} \int_\Omega |\Tr(w_j)(x)|^{q-1}|\Tr(w_j-w_\infty)(x)|dx \\
                                                &\leq \varepsilon\left\|\Tr(w_j)\right\|_{L^{2}(\Omega)} \left\|\Tr(w_j-w_\infty)\right\|_{L^{2}(\Omega)}\\ & \quad+ m_\varepsilon \left\|\beta\right\|_{L^{\infty}(\Omega)} \left\|\Tr(w_j)\right\|_{L^{q}(\Omega)}^{q-1}\left\|\Tr(w_j-w_\infty)\right\|_{L^{q}(\Omega)}.
\end{align*}

Since $\varepsilon$ is arbitrary and the embedding $\Tr(X^{1/2}_0(\mathcal{C}_\Omega))\hookrightarrow L^q(\Omega)$  is compact, we obtain
\indent \begin{eqnarray}\label{j3}
I=\int_{\Omega}\beta(x)|f(\Tr(w_j)(x))||\Tr(w_j-w_\infty)(x)|dx\to 0,
\end{eqnarray}
as $j\rightarrow +\infty$.

Relations \eqref{jj}, \eqref{j2} and \eqref{j3} yield
\begin{eqnarray}\label{fin}
\langle \Phi'(w_{j}),w_{j}-w_\infty\rangle
\rightarrow 0,
\end{eqnarray}
as $j\rightarrow +\infty$ and hence
\begin{equation}\label{fin2}
\displaystyle\int_{\mathcal{C}_\Omega}|\nabla w_j(x,y)|^2\,dxdy
- \int_{\mathcal{C}_\Omega}\langle \nabla w_j,\nabla w_\infty \rangle dxdy\rightarrow 0,
\end{equation}
as $j\rightarrow +\infty$.

Thus, it follows by \eqref{fin2} and \eqref{conv0} that
$$
\lim_{j\rightarrow +\infty}\displaystyle\int_{\mathcal{C}_\Omega}|\nabla w_j(x,y)|^2\,dxdy = \displaystyle\int_{\mathcal{C}_\Omega}|\nabla w_\infty(x,y)|^2\,dxdy.
$$

In conclusion, thanks to \cite[Proposition III.30]{brezis}, $w_j\rightarrow w_\infty$ in $X^{1/2}_0(\mathcal{C}_\Omega)$ and the proof is complete.
\end{proof}

The following technical lemma will be useful in the proof of our result via minimization procedure.

\begin{lemma}\label{lemma1}
Let $f:\RR\to \RR$ be a continuous function satisfying condition~\eqref{Segno}. Then there exists $\widetilde{w}\in X^{1/2}_0(\mathcal{C}_\Omega)$ such that
$
\Psi(\widetilde{w})>0.
$
\end{lemma}

\begin{proof}
Fix a point $x_0\in \Omega$ and choose $\tau>0$ in such a way that
 $$
 B(x_0,\tau):=\{x\in\RR^n:|x-x_0|< \tau\}\subseteq \Omega,
 $$
 where $|\cdot|$ denotes the usual Euclidean norm in $\RR^n$.
 By condition~\eqref{Segno}
 \begin{equation}\label{condF}
\mbox{there exists}\,\,\bar t\in \RR\,\, \mbox{such that}\,\, F(\bar t)>0\,.
\end{equation}
 \indent Hence, let $\bar t\in \RR$ be as in condition~\eqref{condF} and fix $\sigma_0\in(0,1)$ for which
\begin{equation}\label{R1}
F(\bar t)\sigma^n_0\essinf_{x\in \Omega}\beta(x)-(1-\sigma^n_0)\max_{|t|\leq |\bar t|}|F(t)|\|\beta\|_{L^{\infty}(\Omega)}>0.
\end{equation}
Note that this choice is admissible thanks to assumption~\eqref{condF}. Let $\widetilde{u}\in C_0^{1}(\Omega)$ be such that
\[
 \widetilde{u}(x):= \left\{
\begin{array}{ll}
0 & \mbox{ if $x \in \bar\Omega \setminus B(x_0,\tau)$} \\\\
\bar t & \mbox{ if $x \in B(x_0,\sigma_0{\tau}),$}
\end{array}
\right.
\]
and
$|\widetilde{u}(x)|\leq |\bar t|$ if $x \in B(x_0,\tau) \setminus B(x_0,\,\sigma_0\tau)$.\par

 \indent Furthermore, let $\widetilde{w}\in X^{1/2}_0(\mathcal{C}_\Omega)$ be such that $\Tr(\widetilde{w})=\widetilde{u}$. We claim that
 \begin{equation}\label{R2}
 \begin{aligned}
 \int_{\Omega} \beta(x)F(\widetilde{u}(x))\,dx & \geq \Big(F(\bar t)\sigma^n_0\essinf_{x\in \Omega}\beta(x)\\
 & \qquad -(1-\sigma^n_0)\max_{|t|\leq |\bar t|}|F(t)|\|\beta\|_{L^{\infty}(\Omega)}\Big)\omega_n\tau^n,
 \end{aligned}
\end{equation}
where
$$
\omega_n:=\frac{\pi^{n/2}}{\displaystyle\Gamma\left(1+\frac{n}{2}\right)},
$$
denotes the measure of the unit ball in the Euclidean space $\R^n$, with
$$
\Gamma(s):=\int_0^{+\infty}z^{s-1}e^{-z}dz, \quad\forall s>0.
$$
\indent For this purpose, first of all, note that
\begin{equation}\label{ubarlimitata}
|\widetilde{u}(x)| \leq |\bar t|\,\,\,\, \mbox{in}\,\,\, \Omega\,.
\end{equation}
Moreover, by the construction of $\widetilde{u}$, \eqref{ubarlimitata} and the fact that $F(0)=0$, it follows that
\begin{equation}\label{e3.2a25551}
\begin{aligned}
\int_{B(x_0,\tau) \setminus B(x_0,\,\sigma_0\tau)} \beta(x)F(\widetilde{u}(x))\,dx & \geq -\int_{B(x_0,\tau) \setminus B(x_0,\,\sigma_0\tau)} \beta(x)|F(\widetilde{u}(x))|\,dx\\
 & \geq -\|\beta\|_{L^{\infty}(\Omega)}\max_{|t|\leq |\bar t|} |F(t)|\int_{B(x_0,\tau) \setminus B(x_0,\,\sigma_0\tau)} \,dx\\
 & = -\displaystyle\|\beta\|_{L^{\infty}(\Omega)}\max_{|t|\leq |\bar t|}|F(t)|(1-\sigma^n_0)\tau^n\omega_n
 \end{aligned}
\end{equation}
and
\begin{equation}\label{add22}
\int_{\Omega \setminus B(x_0,\tau)} \beta(x)F(\widetilde{u}(x))\,dx =0\,.
\end{equation}
Consequently, relations \eqref{e3.2a25551} and \eqref{add22} and again the definition of $\widetilde{u}$ yield
\begin{equation*}\label{add33}
\begin{aligned}
&\int_\Omega \beta(x)F(\widetilde{u}(x))\,dx\\& = \int_{B(x_0,\,\sigma_0\tau)} \beta(x)F(\widetilde{u}(x))\,dx+\int_{B(x_0,\tau) \setminus B(x_0,\,\sigma_0\tau)} \beta(x)F(\widetilde{u}(x))\,dx\\
& = \int_{B(x_0,\,\sigma_0\tau)} \beta(x)F(\bar t)\,dx+\int_{B(x_0,\tau) \setminus B(x_0,\,\sigma_0\tau)}\beta(x) F( \widetilde{u}(x))\,dx\\
& \geq  F(\bar t)\sigma^n_0\tau^n\omega_n\essinf_{x\in \Omega}\beta(x)-\max_{|t|\leq |\bar t|}|F(t)|(1-\sigma^n_0)\tau^n\omega_n\|\beta\|_{L^{\infty}(\Omega)}\\
& = \left(F(\bar t)\sigma^n_0\essinf_{x\in \Omega}\beta(x)-(1-\sigma^n_0)\max_{|t|\leq |\bar t|}|F(t)|\|\beta\|_{L^{\infty}(\Omega)}\right)\omega_n\tau^n\\
& >0.
\end{aligned}
\end{equation*}
thanks to \eqref{R1}. Clearly, this completes the proof of Lemma~\ref{lemma1}.
\end{proof}

Now, let us prove item $(ii)$ of Theorem \ref{generalkernel0f}.
 
\subsection{First solution via direct minimization}
The assumptions on $\Omega,$ $\beta$ and \eqref{Segno} imply that there exists a suitable smooth function $\widetilde{w}\in X^{1/2}_0(\mathcal{C}_\Omega)\setminus\{0\}$ such that $\Psi(\widetilde{w})>0$, where $\widetilde{u}=\Tr(\widetilde{w})\in H_0^{1/2}(\Omega)$ (see Lemma \ref{lemma1}), and thus the number
\begin{equation}\label{hatlambda}
\lambda^\star:= \inf_{\underset{w\in X^{1/2}_0(\mathcal{C}_\Omega)}{\Psi(w)>0}}\frac{\Phi(w)}{\Psi(w)}
\end{equation}
is well-defined and, in the light of \eqref{sottoquadracita}, positive and finite.\par
 Fixing $\lambda>\lambda^\star$ and choosing $w_\lambda^\star\in X_0$ with $\Psi(w_\lambda^\star)>0$ and
$$
\lambda^\star \leq \frac{\Phi(w_\lambda^\star)}{\Psi(w_\lambda^\star)}<\lambda,
$$
one has
$$
c_{1,\lambda}:= \inf_{w\in X^{1/2}_0(\mathcal{C}_\Omega)} \mathcal J_\lambda(w) \leq \mathcal J_\lambda(w_\lambda^\star)<0.
$$
\indent Since $\mathcal J_\lambda$ is bounded from below and satisfies $(\rm PS)_{c_{1,\lambda}}$, it follows that  $c_{1,\lambda}$ is a critical value of $\mathcal J_\lambda$, to wit, there exists $w_{1,\lambda}\in X^{1/2}_0(\mathcal{C}_\Omega)\setminus\{0\}$ such that
$$
\mathcal J_\lambda(w_{1,\lambda})=c_{1,\lambda} <0 \quad \mbox{ and } \quad \mathcal J'_\lambda(w_{1,\lambda})=0.
$$
This is the first solution we have been searching for.
 
\subsection{Second solution via MPT}
The non-local analysis that we perform in this paper
in order to use the
Mountain Pass Theorem is quite general and may be suitable for other goals,
too.
Our proof will check that the classical geometry of the Mountain Pass Theorem
is respected by the non-local framework.
Fix $\lambda>\lambda^\star$, $\lambda^\star$ defined in \eqref{hatlambda}, and apply \eqref{stimavaloreassolutof} with $\varepsilon:=1/(2\lambda c_2^2)$. For each $w\in X^{1/2}_0(\mathcal{C}_\Omega)$ one has
\begin{align*}
\mathcal J_\lambda(w) & = \frac{1}{2}\left\|w\right\|^2_{X^{1/2}_0(\mathcal{C}_\Omega)} -\lambda\Psi(w)\\
                   & \geq \frac{1}{2}\left\|w\right\|^2_{X^{1/2}_0(\mathcal{C}_\Omega)} - \frac{\lambda}{2}\varepsilon\left\|\Tr(w)\right\|_2^2 - \frac{\lambda}{q}\left\|\beta\right\|_{L^{\infty}(\Omega)} m_\lambda\left\|\Tr(w)\right\|_{L^{q}(\Omega)}^q\\
									 & \geq \frac{1-\lambda\varepsilon c_2^2}{2} \left\|w\right\|^2_{ X^{1/2}_0(\mathcal{C}_\Omega)} - \frac{\lambda}{q}\left\|\beta\right\|_{L^{\infty}(\Omega)} m_\lambda c_q^q\left\|w\right\|^q_{ X^{1/2}_0(\mathcal{C}_\Omega)}.
\end{align*}
Setting
$$
r_\lambda:= \min\left\{\left\|w_\lambda^\star\right\|_{X^{1/2}_0(\mathcal{C}_\Omega)}, \left(\frac{q}{4\lambda\left\|\beta\right\|_{L^{\infty}(\Omega)} m_\lambda c_q^q}\right)^{1/(q-2)}\right\},
$$
due to what has been seen before one has
$$
\inf_{\left\|w\right\|_{X^{1/2}_0(\mathcal{C}_\Omega)}=r_\lambda}\mathcal J_\lambda (w) > 0 = \mathcal J_\lambda(0) > \mathcal J_\lambda(w_\lambda^\star),
$$
namely the energy functional possesses the usual mountain pass geometry.\par
 Therefore, invoking also Lemma \ref{cps}, we can apply the Mountain Pass Theorem to deduce the existence of $w_{2,\lambda}\in X_0$ so that $\mathcal J'_\lambda (w_{2,\lambda})=0$ and $\mathcal J_\lambda (w_{2,\lambda})=c_{2,\lambda}$, where $c_{2,\lambda}$ has the well-known characterization:
$$
c_{2,\lambda}:=\inf_{\gamma\in\Gamma}\max_{t\in[0,1]}\mathcal J_\lambda (\gamma(t)),
$$
where
$$
\Gamma:=\left\{\gamma\in C^0([0,1];{X^{1/2}_0(\mathcal{C}_\Omega)}): \gamma(0)=0, \gamma(1)=w_\lambda^\star\right\}.
$$
\indent Since $$c_{2,\lambda}\geq \inf_{\left\|w\right\|_{X^{1/2}_0(\mathcal{C}_\Omega)}=r_\lambda}\mathcal J_\lambda (w)>0,$$ we have $0\neq w_{2,\lambda} \neq w_{1,\lambda}$ and the existence of two distinct non-trivial weak solutions to \eqref{2} is proved. In conclusion, $\Tr(w_{2,\lambda})$ and $\Tr(w_{1,\lambda})$ are two distinct non-trivial weak solutions to \eqref{problema}.\par
 Furthermore, by \cite[Theorem 5.2]{cabretan}, since \eqref{crescita} holds in addition to $\beta\in L^{\infty}(\Omega)$, it follows that $u_{i,\lambda}:=\Tr(w_{i,\lambda})\in L^{\infty}(\Omega)$, with $i\in\{1,2\}$.
The proof is now complete.

\begin{remark}\label{bound}
\rm{The proof of Theorem \ref{generalkernel0f} gives an exact, but quite involved form of the parameter $\lambda^{\star}$. In particular, we notice that
\begin{equation}\label{hatlambda2}
\lambda^\star:= \inf_{\underset{w\in X^{1/2}_0(\mathcal{C}_\Omega)}{\Psi(w)>0}}\frac{\Phi(w)}{\Psi(w)}\geq \frac{\lambda_1^{1/2}}{c_f\left\|\beta\right\|_{L^{\infty}(\Omega)}}.
\end{equation}
Indeed, by \eqref{c6}, one clearly has
$$
|F(t)|\leq \frac{c_f}{2}|t|^{2},\quad \forall\,t\in \R.
$$
Moreover, since
$$
\|\Tr(w)\|_{L^{2}(\Omega)}^2\leq \frac{1}{\lambda_1^{1/2}}\|w\|^2_{X^{1/2}_0(\mathcal{C}_\Omega)},\quad\forall\, w\in X^{1/2}_0(\mathcal{C}_\Omega)
$$
it follows that
\begin{align*}
\Psi(w) & \leq \int_\Omega \beta(x)|F(\Tr(w)(x))|dx\\
                   & \leq c_f\frac{\|\beta\|_{L^{\infty}(\Omega)}}{2} \|\Tr(w)\|_{L^{2}(\Omega)}^2 \\
                   & \leq c_f\frac{\|\beta\|_{L^{\infty}(\Omega)}}{2\lambda_1^{1/2}}\|w\|^2_{X^{1/2}_0(\mathcal{C}_\Omega)},
\end{align*}
for every $w\in X^{1/2}_0(\mathcal{C}_\Omega)$. Hence, inequality \eqref{hatlambda2} immediately holds. We point out that no information is available concerning the number of solutions of problem \eqref{problema} if
$$
\lambda\in\left[\frac{\lambda_1^{1/2}}{c_f\left\|\beta\right\|_{L^{\infty}(\Omega)}},\lambda^{\star}\right].
$$
\indent Since the expression of $\lambda^{\star}$
is quite involved, we give in the sequel an upper estimate of it which can be easily
calculated. This fact can be done in terms of the same analytical and geometrical constants.
To this end we fix an element $x_0\in \Omega$ and choose $\tau>0$ in such a way that
\begin{equation}\label{tiau}
{B}(x_0,\tau):=\{x\in\R^n:|x-x_0|< \tau\}\subseteq\Omega.
\end{equation}
Now, let $\sigma\in \mathopen{(}0,1\mathclose{)}$, $t\in \R$ and define $\omega_{\sigma}^{t}:\Omega \to \R$ as follows:
\begin{equation*}\label{deftruncfunc}
\omega_\sigma^{t}(x):=
\left\{
\begin{array}{ll}
0 & \mbox{ if $x \in \bar\Omega \setminus B(x_0,\tau)$} \\
\displaystyle\frac{t}{(1-\sigma)\tau} \left(\tau- |x-x_0| \right)
& \mbox{ if $x \in B(x_0,\tau) \setminus B(x_0,\sigma\tau)$} \\
t & \mbox{ if $x \in B(x_0,\sigma{\tau}).$}
\end{array}
\right.
\end{equation*}
\indent It is easily seen that
\begin{equation}\label{norma}
\begin{aligned}
  \int_\Omega |\nabla \omega^{t}_\sigma(x)|^2\,dx & = \int_{B(x_0,\tau)\setminus B(x_0,\sigma\tau)}\frac{t^2}{(1-\sigma)^2\tau^2}dx \\
                                & = \frac{t^2}{(1-\sigma)^2\tau^2} (|B(x_0,\tau)|-|B(x_0,\sigma\tau)|)\\
                                & = \frac{t^2 \omega_n\tau^{n-2}(1-\sigma^n)}{(1-\sigma)^2}.
\end{aligned}
\end{equation}
\indent Let
$$
w_{\sigma}^{t}(x,y):=e^{-\frac{y}{2}}\omega_\sigma^{t}(x),\quad\, \forall (x,y)\in \mathcal{C}_\Omega.
$$
Clearly, $w_{\sigma}^{t}\in X_0^{1/2}(\mathcal{C}_\Omega)$ and, since
$$
|\nabla w_{\sigma}^{t}(x,y)|^2=e^{-y}|\nabla \omega_{\sigma}^{t}(x)|^2+\frac{1}{4}e^{-y}|\omega_{\sigma}^{t}(x)|^2,\quad\, \forall (x,y)\in \mathcal{C}_\Omega
$$
it follows that
\begin{equation}\label{norma2}
\begin{aligned}
\left\|w_{\sigma}^{t}\right\|^2_{X^{1/2}_0(\mathcal{C}_\Omega)} & := \int_{\mathcal{C}_\Omega}|\nabla w_{\sigma}^{t}(x,y)|^2\,dxdy\\
                   & = \int_{\mathcal{C}_\Omega}e^{-y}|\nabla \omega^{t}_\sigma(x)|^2\,dxdy+\frac{1}{4} \int_{\mathcal{C}_\Omega}e^{-y}|\omega^{t}_\sigma(x)|^2\,dxdy\\
                   & =\int_0^{+\infty}e^{-y}dy \left(\int_{\Omega}|\nabla \omega^{t}_\sigma(x)|^2\,dx+\frac{1}{4} \int_{\Omega}|\omega^{t}_\sigma(x)|^2\,dx\right)\\
                   & \leq \int_\Omega |\nabla\omega^{t}_\sigma(x)|^2\,dx+\frac{t^2}{4}|\Omega|.\\
\end{aligned}
\end{equation}
\indent Thus inequalities \eqref{norma} and \eqref{norma2} yield
\begin{equation}\label{norma3}
\left\|w_{\sigma}^{t}\right\|^2_{X^{1/2}_0(\mathcal{C}_\Omega)}\leq \left(\frac{\omega_n\tau^{n-2}(1-\sigma^n)}{(1-\sigma)^2}+\frac{|\Omega|}{4}\right)t^2.
\end{equation}
Moreover, arguing as in Lemma \ref{lemma1}, we have that there exist $t_0\in \R$ and $\sigma_0\in\mathopen{(}0,1\mathclose{)}$ such that
\begin{equation}\label{R24}
\begin{aligned}
\int_{\Omega} \beta(x)F(\Tr(w_{\sigma_0}^{t_0})(x))\,dx & \geq \Bigg(F(t_0)\sigma^n_0\essinf_{x\in \Omega}\beta(x)\\
& \qquad\qquad- (1-\sigma^n_0)\max_{|t|\leq |t_0|}|F(t)|\|\beta\|_{L^{\infty}(\Omega)}\Bigg)\omega_n\tau^n,
\end{aligned}
\end{equation}
with
$$
F(t_0)\sigma^n_0\essinf_{x\in \Omega}\beta(x)-(1-\sigma^n_0)\max_{|t|\leq |t_0|}|F(t)|\|\beta\|_{L^{\infty}(\Omega)}>0.
$$
\indent Due to \eqref{hatlambda} one has
$$
\lambda^\star \leq \frac{\Phi(w_{\sigma_0}^{t_0})}{\Psi(w_{\sigma_0}^{t_0})}.
$$
More precisely, inequalities \eqref{norma3} and \eqref{R24} yield $\lambda^{\star}\leq \lambda_0$, where
$$
\lambda_{0}:=\frac{t_0^2\displaystyle\left(\frac{\omega_n\tau^{n-2}(1-\sigma^n_0)}{(1-\sigma_0)^2}+\frac{|\Omega|}{4}\right)}{2\omega_n\tau^n\left(F(t_0)\sigma^n_0\displaystyle\essinf_{x\in \Omega}\beta(x)-(1-\sigma^n_0)\max_{|t|\leq |t_0|}|F(t)|\|\beta\|_{L^{\infty}(\Omega)}\right)}.
$$
Thus the conclusions of Theorem \ref{generalkernel0f} are valid for every $\lambda>\lambda_0$.
}
\end{remark}
\noindent\textit{Proof of Theorem} \ref{boccia}.
For any $t\in\R$, set
$$
F_+(t):=\int_0^t  f_+(z)dz,
$$
with
$$
f_+(z):=\left\{
\begin{array}{ll}
f(z)  & \mbox{ if } z\geq 0\\
0 & \mbox{ if } z<0.
\end{array}
\right.
$$
and define in a natural way $\mathcal J_\lambda^+: X^{1/2}_0(\mathcal{C}_\Omega)\to\R$ to be
$$
\mathcal J_\lambda^+(w):=\Phi(w)-\lambda\Psi_+(w),
$$
for any $u\in X^{1/2}_0(\mathcal{C}_\Omega)$, with
$$
\Psi_+(w):= \int_{\Omega} \beta(x) F_+(\Tr (w)(x)) dx.
$$

It is easy to see that the functional $\Psi_+$ is well-defined and Fr\'echet differentiable at any $u\in X^{1/2}_0(\mathcal{C}_\Omega)$ (being $F_+$ differentiable in $\R$) and that Theorem \ref{generalkernel0f} holds replacing $f$ by $f_+$. As a result (by using the Strong Maximum Principle \cite[Remark 4.2]{cs}) there exist two (positive) distinct critical points of $\mathcal J_\lambda^+$. Now, set
$$
S:=\{t>0:F(t)>0\}\,\,\,{\rm and}\,\,\,z_n(\sigma):=\frac{1-\sigma^n}{(2\sigma^n-1)(1-\sigma)^2},\,\,\,\forall\sigma\in \Sigma_n:=\left(\frac{1}{2^{1/n}},1\right).
$$
By hypotheses \eqref{supf}--\eqref{Segno} it follows that there exists $t_0>0$ such that
\begin{equation}\label{lambda12}
\frac{t^2_0}{F(t_0)}=\min_{t\in S}\frac{t^2}{F(t)}>0.
\end{equation}
\indent On the other hand, bearing in mind that $f$ is non-negative, owing to
$$
\lim_{\sigma\rightarrow \frac{1}{2^{1/n}}^+} z_n(\sigma)=\lim_{\sigma\rightarrow 1^-} z_n(\sigma)=+\infty,
$$
there exists $\sigma_0\in \Sigma_n$ such that
\begin{equation}\label{lambda1234}
F(t_0)(2\sigma_0^n-1)=\left(F(t_0)\sigma^n_0-(1-\sigma^n_0)\max_{|t|\leq t_0}F(t)\right)>0.
\end{equation}
\indent  Then by Remark \ref{bound}, inequalities \eqref{lambda12} and \eqref{lambda1234} ensure that for every
\begin{equation}\label{lambda}
\lambda>\frac{1}{2}\left(\frac{1}{r^2}\min_{\sigma\in \Sigma_n}z_n(\sigma)+\frac{1}{4}\right)\min_{t\in S}\frac{t^2}{F(t)},
\end{equation}
the following nonlocal problem
\begin{equation*}\label{problema241}
\left\{
\begin{array}{ll}
A_{1/2}u=\lambda f(u) & \mbox{\rm in } \Gamma^0_r\\
u> 0 & \mbox{\rm on } \Gamma^0_r\\
u=0 & \mbox{\rm on } \partial\Gamma^0_r,
\end{array}\right.
\end{equation*}
admits at least two distinct and nontrivial weak solutions $u_{1,\lambda}, u_{2,\lambda}\in L^{\infty}(\Gamma^0_r) \cap H_0^{1/2}(\Gamma^0_r)$.\par
Since condition \eqref{disball} holds, inequality \eqref{lambda} is satisfied for $\lambda=1$. Hence, problem \eqref{problema24} admits at least two distinct $L^{\infty}$-bounded weak solutions.\qed

\smallskip
\indent In conclusion, we present a direct application of our main result.

\begin{example}\label{esempio}
Let $\Omega$ be an open bounded set of
$\R^n$ ($n\geq 2$) with Lipschitz boundary $\partial\Omega$.
As a model for $f$ we can take the nonlinearity
$$
f(t):=\log(1+t^2),\,\quad\forall\,t\in\R.
$$
Indeed, the real function $f$ fulfills hypotheses \eqref{supf}--\eqref{Segno}. Hence, Theorem \ref{generalkernel0f} and Remark \ref{bound} ensure that for every
$$
\lambda>\frac{1}{2\omega_n\tau^n}\left(\omega_n\tau^{n-2}\min_{\sigma\in \Sigma_n}z_n(\sigma)+\frac{|\Omega|}{4}\right)\min_{t>0}\left(\frac{t^2}{2\arctan t+t\log(1+t^2)-2t}\right),
$$
the nonlocal problem
\begin{equation*}\label{problema2}
\left\{
\begin{array}{ll}
A_{1/2}u=\lambda \log(1+u^2) & \mbox{ in } \Omega\\
u=0 & \mbox{ on } \partial\Omega,
\end{array}\right.
\end{equation*}
admits at least two distinct weak solutions $u_{1,\lambda}, u_{2,\lambda}\in L^{\infty}(\Omega) \cap H_0^{1/2}(\Omega)\setminus\{0\}$.
\end{example}

\begin{remark}\rm{We conclude by recalling that a similar variational approach as we have employed has been extensively
used in several contexts, in order to prove multiplicity results of different
problems, such as elliptic problems on either bounded or unbounded domains of the Euclidean space (see \cite{k0,k2,k3,k5}), elliptic equations involving the Laplace-Beltrami operator on Riemannian manifold (see \cite{k1}), and, more recently, elliptic equations on the ball endowed with Funk-type metrics \cite{k4}. See also \cite{MRadu1}, where
a multiplicity result analogous to the one proved in the present paper
is considered when the underlying operator is the nonlocal one studied in \cite{BaMo, MoReSe,svmountain}.}
\end{remark}

\section*{Acknowledgments}
 This research was done under the auspices of the INdAM - GNAMPA Project 2016 entitled: {\it Problemi variazionali su variet\`a Riemanniane e gruppi di Carnot}, of the INdAM - GNAMPA Project 2017 titled: {\it Teoria e Modelli non-locali}, and the SRA grants P1-0292, J1-7025, J1-8131, and N1-0064.
 The first and the second author was partly supported
by the Italian MIUR project
{\em Variational methods, with applications to problems in mathematical
physics and geometry} (2015KB9WPT\_009).
The authors warmly thank the anonymous referees for their useful and nice
comments on the paper.

\end{document}